\documentclass{amsart}
\usepackage{amsmath,amsthm,amssymb,xypic}
\usepackage[all]{xy}
\theoremstyle{plain}

\newtheorem{theorem}{Theorem}[section]

\newtheorem{lemma}[theorem]{Lemma}
\newtheorem{corollary}[theorem]{Corollary}

\newtheorem{example}[theorem]{Example}

\theoremstyle{definition}

\newtheorem{definition}[theorem]{Definition}

\theoremstyle{remark}

\newtheorem{claim}{Claim}

\newtheorem{con}[theorem]{Construction}

\newtheorem{remark}[theorem]{Remark}

 \DeclareMathOperator{\cod}{cod}

\newcommand{\QED}{\ifhmode\unskip\nobreak\fi\quad {\rm Q.E.D.}} 

\newcommand\Span[1]{\langle{#1}\rangle}

\newcommand{\C}{\mathbb{C}}

\newcommand{\I}{\mathcal{I}}

\renewcommand{\L}{\mathcal{L}}

\renewcommand{\O}{\mathcal{O}}
\renewcommand{\o}{\mathcal{O}}
\renewcommand{\P}{\mathbb{P}}

\newcommand{\p}{\mathbb{P}}

\newcommand{\Z}{\mathbb{Z}}

\newcommand{\rat}{\dasharrow}

\newcommand{\dJen}{de Jonqui\'eres\ }

\begin{document}

\title{Equivalent birational embedding IV: reduced  varieties}

\author{
Massimiliano Mella}
\address{
  Dipartimento di Matematica\\ Universit\`a di Ferrara\\
  Via Machiavelli 30, 44121 Ferrara
 Italia}
\email{mll@unife.it} 

\date{November 2021}
\thanks{The author is a member of GNSAGA}
\maketitle
\begin{abstract} Two reduced projective schemes are said to be
  Cremona equivalent if there is a Cremona map that maps one in the
  other. In this paper I revise some of the known results about
  Cremona equivalence and extend the main result of \cite{MP09} to
  reducible schemes. This allows to prove a very
  general contractibility result for union of rational subvarieties.
\end{abstract}

\section{Introduction}
The birational geometry of the projective space has always attracted the
attention of algebraic geometers.  The Cremona group, $Cr(\p^r_k)$, that is the group
of birational selfmaps of the projective space, has been intensively
studied for well over a century but it is still a quite
mysterious object.
Here is an extract from the article ``Cremona group'' in the
Encyclopedia of Mathematics, written by V. Iskovskikh in 1982:

{\vspace{0,5cm}
\it One of the most difficult problems in birational geometry is that of
describing the structure of the group $Cr(\p^3_k)$, which is no longer
generated by the quadratic transformations. Almost all literature on
Cremona transformations of three-dimensional space is devoted to
concrete examples of such transformations. Finally, practically
nothing is known about the structure of the Cremona group for spaces
of dimension higher than 3.}\cite{Is87b}.
\vspace{0,5cm}

Unfortunately after  $40$ years the situation is not much
better. A reasonable set of generators is not known yet. The Cremona groups have been proved to be non simple,
\cite{CL13} \cite{BLZ19}, and their behavior is wild from many points of view, as an
example one can look at the results in \cite{BLZ19}. Already the two
dimensional Cremona group has  many foundational problems that are far
from being  solved, see
 \cite{Ca15} for a very nice introduction.
Instead of trying to tame this group, in this paper I want to use its
wildness  to address the 
following question.

\vspace{.5cm}
{\it Question}: Let $X,Y\subset \p^r$ be birational reduced schemes is
there a birational selfmap of the
projective space $\omega:\p^r\dasharrow\p^r$ such that
$\omega(X)=Y$?
\vspace{.5cm}
 
When the answer to the question is positive $X$ and $Y$ are said to be
Cremona equivalent.
The notion of Cremona equivalence is quite old and already at the end
of $\rm XIX^{th}$ century both Italian and English school of algebraic
geometry approached the problem, with special regards to plane
curves, \cite{CE00} \cite{Ma07}  \cite{HU27} \cite{Co59}.   
The first result I am aware off in modern times, is due to Jelonek, \cite{Je87} where
the author proved that two irreducible and reduced birational subvarieties of the
complex projective space are Cremona equivalent when roughly  the
dimension is smaller than the codimension. More recently I have been attracted by the problem and
dedicated a series of papers to explore the possibility to extend
birational maps of projective varieties to the ambient space. The first
important improvement has been achieved in \cite{MP09}, see also
\cite{CCMRZ16} for an alternative proof, where it is proven that two
irreducible and reduced birational projective
varieties of codimension at least 2 are Cremona equivalent.

It is not difficult to see that the result is sharp with respect to
the codimension. It was classically known the existence of non Cremona
equivalent rational  plane curves, see for instance
Example~\ref{exa:sextic}.
Then stemming from the mentioned result in \cite{MP09} there are
two possible directions: study the Cremona equivalence of divisors,
extend the result to reducible and reduced projective varieties.
The case of divisors has been fruitfully studied.  In \cite{MP12} and
\cite{CC10} the authors completely described the Cremona equivalence
classes of irreducible curves and gave conditions for
a plane curve to be of minimal Cremona degree, that is with the
smallest degree in the Cremona equivalence class. Partial results have
also been obtained for special classes of divisors, \cite{Me13}, and
special classes of rational surface, \cite{Me20} \cite{Me21}.

To the best of my knowledge the only cases of reducible varieties
studied in relation to the  Cremona equivalence
 are those concerning the
contractibility of set of lines in the projective plane, \cite{CE00}  \cite{CC17} \cite{CC18} \cite{Du19}.
Even for this very special class of reducible varieties the answer
is really complicate and it is not known yet a complete classification
of contractible set of lines in the plane. The problem resting on the
different possible configuration of intersection points. Note further
that essentially nothing is known about the Cremona equivalence class
of non contractible set of lines in the plane. That is given two
configurations of lines in the plane nothing is known about their
Cremona equivalence.

In this paper I want to address the case of reduced schemes of
codimension at least $2$. I was really amazed when I realized that
also for this class  it is possible to extend the result of
irreducible subvarieties and give a complete answer to the question.

\begin{theorem}\label{th:main}
  Let $X,Y\subset\p^r$ be two reduced
  schemes of dimension at most $r-2$. Then $X$ is Cremona equivalent to $Y$ if and only if
  $X$ and $Y$ are birational. 
\end{theorem}

To appreciate the result and the amazing flexibility of the Cremona
group just think of a bunch of lines $L=\cup_1^s l_i\subset\p^r$, for $n\geq 3$. Then
there is a birational map $\omega:\p^r\dasharrow \p^r$ mapping $L$ to
a set of $s$ lines passing through a fixed point $p$. Hence any set of
lines in $\p^n$ is contractible by a birational map as soon as $n\geq
3$.
No matter how the irreducible components of $L$ intersects we can
always contract them to a set of $s$ points with a birational selfmap of $\p^r$. As an
application of Theorem~\ref{th:main} we'll prove a similar statement for an
arbitrary set of rational varieties.

Despite the proof of the theorem is constructive and algorithmic, it is difficult if
not impossible to produce a birational map  that realizes the Cremona
equivalence of two prescribed subvarieties. This is due to the fact that to follow the proof's steps
one has to produce irreducible monoids with special features and those are
difficult to be computed on effective examples.
Coming back to the wildness of $Cr(\P^r_k)$ the positive answer to the
Cremona equivalence question for arbitrary subvarieties of
codimension at least $2$ can be seen as a further confirmation of
the difficulty to describe and tame this incredible group of
transformations, see also Remark~\ref{rem:wild}.

The paper is organized as follows. First I  introduce a set of special
Cremona birational maps and use them to study the Cremona equivalence
of special classes of varieties, in particular an explicit
construction of the Cremona equivalence of sets of reduced points is
given.
Even if this part is not strictly necessary
to prove the main result I think it allows to perceive the beauty of
Cremona modification and it is also a nice training camp on the birational
geometry of projective subvarieties of the projective space.
The proof of the  main Theorem is then finished in the last section. To do it  I adapt
the
proof in \cite{CCMRZ16} to the case of reducible varieties. This is
done improving the computation of the dimension  of monoids containing a
subvariety, see Lemma~\ref{lem:mp16_1}, and avoiding the use of  the
results in \cite{CC01} on the Segre locus to produce the chain of
double projection needed to complete the argument.
Finally as an application it is proven that any set of reduced
codimension at least two rational varieties can be contracted by a
Cremona transformation.

Many thanks are due to the referee for a careful reading and for
suggesting an improvement to Lemma~\ref{lem:moniddim}.

\section{Basics on Cremona transformations} 
We work over the complex field.
\begin{definition}\label{def:Cremona}
  A \emph{Cremona transformation} is a birational map
  $\varphi:\p^r\rat\p^r$ given by equations
 \begin{equation*}
   \label{eq:1}
 [x_0,\ldots,x_r]\mapsto [F_0(x_0,\ldots,x_r),\ldots, 
 F_r(x_0,\ldots,x_r)],
\end{equation*}
where $F_i(x_0,\ldots,x_r)$ are homogeneous
polynomials of the same degree $\delta>0$, for $i=0,\ldots,r$.

 The
inverse map is also a Cremona transformation, and it is defined by
homogeneous polynomials $G_i(x_0,\ldots,x_r)$ of degree $\delta'>0$,
for $i=0,\ldots, r$.
If the polynomials $\{F_i\}$  are coprime and we choose the $\{G_j\}$
as well coprime  we say that
$\varphi$ is a \emph{$(\delta, \delta')$--Cremona transformation}.

The subscheme 
$${\rm Ind}(\varphi):=\cap^r_{i=0}(F_i(x_0,\ldots,x_r)=0)$$
is the {\em indeterminacy locus} of $\varphi$.  Since the
composition of $\varphi$ and its inverse is the identity, we have
\[G_i(F_0(x_0,\ldots,x_r), \ldots, F_r(x_0,\ldots,x_r))=\Phi\cdot
  x_i,\ \ \mbox{for $i=0,\ldots, r$}\]
where $\Phi$ is a
homogeneous polynomial of degree $\delta\cdot \delta'-1$. The
hypersurface ${\rm Fund}(\varphi):=\{\Phi=0\}$ is the
\emph{fundamental locus} of $\varphi$ and its support is the
\emph{reduced fundamental locus} ${\rm Fund}_{\rm red}(\varphi)$. 
The group of Cremona transformation of $\p^r$ is
$$Cr_r:=:Cr(\p^r):=\{\varphi:\p^r\rat\p^r| \mbox{ the map $\varphi$ is birational}\}. $$
\end{definition}
\begin{remark}
  Note that we are not asking for the polynomials $\{F_i\}$ to be
  coprime. This is quite unusual but useful to prove the main result,
  indeed this allow us to add fixed components to linear system to
  produce the birational maps we need. 
\end{remark}
Let us work out some special cases in details.

\begin{example}[Quadro-quadric transformation of $\p^r$]
  \label{exa:quadro-quadric} Let $Q\subset
  H\subset\P^r$ be a codimension $2$ reduced quadric and
  $p\in\p^r\setminus H$ a point. Consider the linear system
  $$\L=|\I_{p\cup Q}(2)|$$
  of quadrics through $p$ and $Q$. Then
  $\dim\L=n$ and the scheme theoretic base locus of $\L$ is $Q\cup
  p$.

  Let $\epsilon:Z\to\p^r$ be the blow up of $Q$ and $p$ with
  exceptional divisor $E_p$ and $E_Q$ and $\nu:Z\to\p^{r+1}$ the blow down of the
  hyperplane $H$ and of the cone $C_p(Q)$ with base  $Q$ and  vertex $p$. Then
  a general conic passing through $p$ and intersecting $Q$ in $2$
  points is mapped to a line. Therefore 
  $\phi:=\nu\circ\epsilon$ is a Cremona transformation.
  For a general
  hyperplane $H$ the restriction $\phi_{|H}$ maps $H$ to a quadric, then  the
  inverse of $\phi$ is again given by quadrics with an isomorphic base
  locus.
 This shows that $\phi$ is a quadro-quadric Cremona transformation.
  
  Note that for $n=2$ the map $\phi_\L$  is the standard quadratic
  Cremona transformation. Moreover for  a general linear space $\p^a\cong
  A\subset\p^r$ containing $p$, the restriction  $\phi_{\L|A}$ is
  again a quadro-quadric map $\p^a\dasharrow\p^a$.
  
Recall that Noether--Castelnuovo Theorem shows that $Cr_2$ is generated by the
linear automorphisms and the 
quadro-quadric transformation of $\p^2$.  Therefore if we consider a
plane $A\cong\p^2$ and any birational map $\omega:A\dasharrow A$ we
may write
$$\omega=g_1\circ\phi_1\circ\ldots\circ g_h\circ\phi_{h},$$
with $\phi_i$ quadro-quadric maps and $g_i$ linear automorphisms of $\p^2$.
If $A\subset\p^r$ we may extend both quadro-quadric maps and linear
automorphisms to selfmaps of the ambient space. Hence for any map
$\omega\in Cr(\p^2)$ there is a birational map $\Omega\in Cr(\p^r)$ such that
$\Omega_{|A}=\omega$, as birational maps.
\end{example}

\begin{example}[Cubo-cubic transformation of $\p^3$]
  \label{exa:cubo-cubic}
  Let $\Gamma\subset\p^3$ be a rational normal curve and $S_1,S_2\in
  |\I_\Gamma(3)|$ two smooth cubic surfaces containing $\Gamma$.  Then
  we have
  $S_1\cap S_2=\Gamma\cup R$, for a residual curve $R$ of degree $6$
  genus $3$. It is not difficult to check that
  $$\dim |\I_R(3)|=3.$$
  and $\Gamma\cdot R=8$, see for instance \cite{Ve05}.
  This shows that the linear system $|\I_R(3)|$ defines a Cremona
  transformation $\psi:\p^3\dasharrow\p^3$.
  that can be described as follows.
  Let $\epsilon:\Z\to\p^3$ be the blow up of $R$ and $\nu:Z\to\p^3$
  the blow down of the strict transform of trisecant lines to
  $R$. Then we have
  $$\psi=\nu\circ\epsilon.$$

  Since $\Gamma\cdot R=8$ we have  that  $\psi(\Gamma)$ is a line, moreover the
  restriction to a general plane $\psi_{|H}$ maps $H$ to $\p^2$ blown
  up in $6$ points, the intersection points with the curve $R$. 
  Therefore the inverse of $\phi$ is again defined by cubics and with
  a bit more of work one can prove that the base locus is of the same
  type.
  In particular $\phi$ is a cubo-cubic Cremona transformation
\end{example}

Next we introduce a class of special hypersurfaces that will be of
crucial importance in what follows.

\begin{definition}[Monoids] Let $X\subset \P^ r$ be a hypersurface of degree
$d$. We say that $X$
is a \emph{monoid} with \emph{vertex} $p\in \P^ r$ if $p$ is a point in $X$
of multiplicity exactly $d-1$. Note that a monoid can have more than one
vertex.  If we choose projective coordinates in such a way that
$p=[1,0,\ldots,0]$, then 
\begin{equation*}\label{eq:monoid}
  X=(F_{d-1}(x_1,\ldots,x_{r})x_{0}+
  F_{d}(x_1,\ldots,x_{r})=0),\end{equation*} 
where $F_{d-1}$ and $F_d$  are homogeneous polynomials of degree $d-1$
and $d$ respectively and $F_{d-1}$ is nonzero. The hypersurface $X$ is
irreducible if and only if  $F_{d-1}$ and $F_d$ are coprime. 
\end{definition}

\begin{con}
An irreducible monoid $X$ is rational. Indeed, the
projection of $X$ from a vertex $p$ onto a hyperplane $H$ not
passing through $p$ is a birational map $\pi\colon X\dasharrow H\cong \P^
{r-1}$. If $H$ has equation $x_0=0$, then the inverse map
$\pi^ {-1}\colon \P^ {r-1} \dasharrow X $ is given by 
\[
[x_1,\ldots, x_{r}] 
\mapsto [-F_{d}(x_0,\ldots,x_{r-1}),{F_{d-1}(x_1,\ldots,x_{r})}x_1,\ldots, {F_{d-1}(x_0,\ldots,x_{r-1})}x_{r}].
\]
The map $\pi$ is called the \emph{stereographic projection} of $X$ from
$p$.  Its indeterminacy locus  is $p$. Each line through $p$ contained in $X$ gets contracted to a point under $\pi$.  The set of all such lines is defined by the
equations $\{F_d=F_{d-1}=0\}$. This is the indeterminacy locus
of $\pi^ {-1}$, whereas the hypersurface of $H$ with equation
$\{x_0=F_{d-1}=0\}$ is contracted to $p$ by the map $\pi^ {-1}$.
\end{con}

Monoids are useful to produce an important class of Cremona transformations.

\begin{example}[\dJen transformations]
  A \dJen transformation of $\p^r$ is a birational map that preserves
  the family of lines through a point, say $p$, see  \cite{PS15} for a
  comprehensive introduction.  Let
  $\omega:\p^r\rat\p^r$ be a \dJen transformation given by
  $$[x_0,\ldots,x_n]\mapsto [M_0,\ldots,M_n].$$
  Up to a linear automorphisms we may assume that the lines through
  $[1,0,\ldots,0]$ are mapped to lines through $[1,0,\ldots,0]$. Then
  we may choose the $X_i:=(M_i=0)$ to be monoid with vertex
  $[1,0,\ldots,0]$. Moreover $\{X_1,\ldots,X_n\}$ has to contain a
  common divisor $B$, which has to be  itself a monoid.
  This shows that a \dJen map, up to  linear automorphisms, is of the form
  $$[x_0,\ldots,x_n]\mapsto
  [x_0F_0+\tilde{G}_0,G_1(x_0F_1+\tilde{G}_1),\ldots,G_n(x_0F_1+\tilde{G}_1)],$$
  where $G_j\in\C[x_1,\ldots,x_n]_g$, for $j=1,\ldots n$, and 
  $$[x_1,\ldots,x_n]\mapsto [G_1,\ldots, G_n]$$
  defines a Cremona transformation of $\p^{n-1}$.

  We will always be concerned with the subclass of \dJen maps where
  $g=1$. That is Cremona transformations that admit  the form
   $$[x_0,\ldots,x_n]\mapsto
  [x_0F_0+\tilde{G}_0,x_1(x_0F_1+\tilde{G}_1),\ldots,x_n(x_0F_1+\tilde{G}_1)],$$
  with a slight abuse of language we will also write $\omega$ as the
  map associated to the linear system
  $$\{x_0F_0+\tilde{G}_0,x_1(x_0F_1+\tilde{G}_1),\ldots,x_n(x_0F_1+\tilde{G}_1)\}.$$
  With these notation we have:
  \begin{itemize}
  \item[-]  the indeterminacy locus of $\omega$ is $(x_0F_0+\tilde{G}_0)\cap (x_0F_1+\tilde{G}_1)$,
    \item[-]   the family of
      lines through $[1,0,\ldots,0]$ is preserved
      \item[-]  $\omega$ is
birational and the inverse is again a \dJen transformation.
  \end{itemize}
To convince you let me write the map in the
following equivalent way
$$[x_0,\ldots,x_n]\mapsto [\frac{x_0F_0+G_0}B,x_1,\ldots,x_n].$$
This shows that the lines through $[1,0,\ldots,0]$ are preserved, the
map is birational and its inverse is of the same form.

Let me stress  that a \dJen transformation restricts to a linear
automorphism on a general line through the special point $p$.
 
The quadro-quadric map described in
Example~\ref{exa:quadro-quadric} is a \dJen of degree 2, where $B$ is
the span of the codimension $2$ quadric.
The case of $\p^2$ is particularly interesting. A \dJen map  of degree
$d$ has $2d-2$ simple base points, maybe infinitely near. The map can
be factored via the blow up of the multiple point and then $2d-2$
elementary transformations of Hirzebruch surfaces to finally contract
the curve $B$. 
\end{example}

\section{Cremona equivalence: definition and first examples}
Let us introduce in details the main relation we are going to analyze.
\begin{definition}
  Let $X,Y\subset\p^N$ be two birational reduced schemes. We say that $X$
  is Cremona equivalent to $Y$ if there is a birational modification
  $\omega:\p^N\dasharrow \p^N$ that is an isomorphism on the generic points of $X$,
  such that $\omega(X)=Y$.
\end{definition}

To get acquainted it is useful to have some examples in mind.
\begin{example}
Let $C\subset\p^3$ be a twisted cubic. Let $S_1, S_2$ be two general
cubic surfaces containing $C$. As we saw in
Example~\ref{exa:cubo-cubic} there is a cubo-cubic modification of
$\p^3$ that maps the $S_i$ to planes and hence $C$ to a line.
So  $C$ is Cremona
equivalent to a line.
\end{example}

The next is again a cubic curve but reducible
\begin{example}\label{exa:lines}
Consider two sets of  three lines in $\p^n$, say $L_1,L_2,L_3$ and
$R_1,R_2,R_3$.
Let us start assuming that $n=2$.  Let $p_{ij}=L_i\cap L_j$ and $q_{hk}=R_h\cap R_k$. Let
$\lambda=\sharp\{p_{ij}\}_{i,j\in\{1,2,3\}}$ and
$\rho=\sharp\{q_{hk}\}_{h,k\in\{1,2,3\}}$ be the cardinality of the
intersection points. If
$\lambda=\rho$ then  there
is a linear automorphism  of $\p^2$ realizing the Cremona
equivalence. Indeed with the choice of $4$ points we can map one
configuration to the other.

Assume, without loss of generality, that $\lambda=1$ and $\rho=3$.
This time we need a birational modification to put the lines
$\{L_1,L_2,L_3\}$ in general position. Let $p_1\in L_1$, $p_2\in L_2$,
$q_1,q_2\in L_3$ and $x\in\p^2$ be general points. Consider the linear system
$$\L=|\I_{p_1^2\cup p_2^2\cup x^2\cup p_{12}\cup q_1\cup q_2}(4)|$$
of quartics singular along $p_1,p_2,x$ and passing through the intersection point
$p_{12}$ and $q_1$, $q_2$.
Then $\phi_\L:\p^2\dasharrow\p^2$ is a Cremona transformation, the
composition of two standard Cremona maps,
$$\deg \phi_\L(L_i)=4-3=1$$
and the lines $\{\phi_\L(L_1),\phi_\L(L_2),\phi_\L(L_3)\}$ are in general position.

We already observed that, thanks to Noether--Castelnuovo Theorem and
the quadro quadric transformation of Example~\ref{exa:quadro-quadric},
any Cremona map of $\p^2$ can be extended to an arbitrary $\p^r$.
Then  for $r\geq 3$ it is enough to prove that  any set of three lines is
Cremona equivalent to a set of three lines in a plane.

Let $\{L_1,L_2,L_3\}\subset\p^r$ be a set of three lines. Assume first
that there is an irreducible quadric $Q\subset\p^r$ containing the set.
Let $Y\subset Q$ be a general hyperplane section and $p\in Q$ a
general point. Then the quadro-quadric map centered in $Y$ and $p$
maps  $\{L_1,L_2,L_3\}$ to a set of three lines in
$\p^{r-1}$. Therefore by a recursive argument we may assume that the
lines $\{L_1,L_2,L_3\}$ are contained neither in an irreducible quadric nor in a
plane and in
particular $n=3$.

Here I want to propose two different approaches. First consider a
point $p\in\p^3$ and 
conic $C\subset\p^3$ intersecting the three lines. Let $X$ be  a quartic
singular along $C\cup p$ and containing $L_1\cup L_2\cup L_3$.
By an easy dimensional count $X$ exists and it is mapped to a quadric
by a quadro-quadric map centered in $C\cup p$. This is enough to prove
that all lines triples are Cremona equivalent.

 Then I want to argue in a different way.  Without
loss of generality, we may assume that $L_1$ is skew to $L_2$ and
$L_3$ and $L_2\cap L_3=p$.
Pick a general cubic surface $S$ containing $L_1, L_2,L_3$ and let
$R\subset S$
be a line intersecting $L_3$ and skew with $L_1$ and $L_2$.  Let
$\pi:S\to\p^2$ be the blow down of $L_1, L_2, R$ and three more
$(-1)$-curves in $S$, to points $p_1,p_2,p_3,p_4,p_5, p_6\in
\p^2$. Then $\pi(L_3)$ is a line in $\p^2$ spanned by $p_2$ and
$p_3$. Let $C\subset\p^2$ be a conic with
$$C\cap\{ p_1,p_2,p_3,p_4,p_5, p_6\}=\{p_1,p_2,p_4\},$$
then $\Gamma:=\pi^{-1}_*C$ is a twisted cubic intersecting $L_1, L_2$ and
$L_3$ in a point. Let $S_1$ be a general cubic surface containing
$\Gamma$ and $R$ the residual intersection
$$S\cap S_1=\Gamma\cup R.$$
Then $R\cap L_i=2$ and the  cubo-cubic map $\phi:\p^3\dasharrow \p^3$
centered in $R$ maps $S$ to a plane and the $L_i$ in lines. This is
enough to conclude that all triples of lines in $\p^n$ are Cremona
equivalent.

Despite the beauty of this constructions it is clearly impossible to
proceed in this way for an arbitrary number of lines. Already four
lines have many different configurations and one should be able to
produce a Cremona modification for all of them. Note further that in
$\p^2$ not all line configurations are Cremona equivalent, see
\cite{Du19} for a vast treatment of lines configuration with respect to
Cremona equivalence and the problem of contractibility.
\end{example}

Next we consider monoids.
\begin{example}\label{exa:monoids}
Irreducible monoids are always Cremona equivalent to a hyperplane. Let
$X\subset\p^n$ be a monoid of degree $d$ with vertex
$p_0=[1,0,\ldots,0]$ and $Y$ a monoid of degree $d-1$ with the same
vertex.
Then the \dJen transformation given by
$$\{X,Yx_1,\ldots,Yx_n\}$$
maps $X$ to the hyperplane $(y_0=0)$. Hence any irreducible monoid is Cremona
equivalent to a hyperplane.
\end{example}

\begin{example} Any irreducible rational surface in $\p^3$ of degree at most $3$ is
  Cremona equivalent to a plane. Quadrics and singular cubics are
  monoids, therefore we conclude with Example~\ref{exa:monoids}. For
  smooth cubic we may use the cubo-cubic map to produce the
  equivalence. Note that non rational irreducible cubics, that is
  cones over elliptic curves, are not Cremona equivalent to any
  surface of lower degree, simply because
  all surfaces of smaller degree are rational.

  Already for quartic surfaces in $\p^3$  the situation is much more
  complicate, but it is still possible to study it, see \cite{Me20} \cite{Me21}.
\end{example}

Despite this quite long list of explicit examples of Cremona
equivalences it is  in general  quite rare to be able
to control birational modification that realizes a Cremona
equivalence. On the other hand the Cremona group is so flexible that it
is able to realize a huge set of  Cremona equivalences.
We are ready to appreciate the following theorem.

\begin{theorem}\label{th:mp}
 \cite{MP09}\cite{CCMRZ16} Let $X,Y\subset \p^r$ be irreducible
 birational subvarieties and assume that $\dim X\leq r-2$. Then $X$ is
 Cremona equivalent to $Y$.
\end{theorem}

Let me spend a few words on this result and its consequences.
The Theorem proves that the Cremona group
contains, as subsets, the set of birational self map of any subvariety
of codimension
at least two. Note that in general nothing can be said on the group
structure. That is there is no hint that it is possible to realize the
group of birational selfmaps of a subvariety as subgroup of some
Cremona group. Despite the proof of Theorem~\ref{th:mp}, especially
the second one, is quite algorithmic it is in general almost impossible
to write down an explicit map  that furnishes the Cremona equivalence.
On the other hand for few special cases of rational varieties it is possible to
describe an explicit linearization, see \cite{CCMRZ16}.

It is quite easy to see that the bound on the codimension is sharp.

\begin{example}\label{exa:sextic}
 Let $C\subset\P^2$ be an irreducible rational curve of degree $6$ with ordinary
 double points. Note that the pair $(\p^2,\frac12 C)$ has canonical
 singularities, therefore, by a standard application of Sarkisov
 theory, \cite{MP09}, any curve Cremona equivalent to $C$ has degree at least $6$,
 therefore $C$ is not Cremona equivalent to a line.  

In a similar fashion it is easy to produce examples in arbitrary
dimension, see \cite{MP09}. It is also possible to see that in
general a fixed abstract variety has infinitely many inequivalent
birational embeddings with respect to Cremona equivalence, \cite{MP09}.
\end{example}

\begin{definition}
  \label{def:cones} A reduced variety $Z\subset \p^n$ is a cone if there is a
  point $p\in Z$, called vertex, such that $Z=\cup_{x\in Z\setminus\{p\}}\langle
  x,p\rangle$. The cone with vertex $p\in\p^n$ and base $X\subset\p^n$
   is
  $$C_p(X):=\cup_{x\in X}\langle x,p\rangle. $$  
\end{definition}

\begin{example}(Cones) In \cite{Me13} it is proven that two divisorial
  cones $X,Y\subset\p^n$ are Cremona equivalent if their general
  hyperplane sections are Cremona equivalent. In particular, thanks to
  Theorem~\ref{th:mp}, a divisorial cone
  over a rational variety is always Cremona equivalent to a hyperplane.
\end{example}

It is less clear if the irreducibility assumption is needed. On one
hand the example of lines, Example~\ref{exa:lines}, is not
encouraging since the Cremona modification needed depends on the
intersection between the irreducible components. On the other hand
there are no theoretical limits to extend the proof to reduced schemes.
I must say that I was  quite surprised
when I realized that with few improvements a combination of the proofs in
\cite{MP09} and \cite{CCMRZ16}  worked in the reducible  case.
Before going into this I want to give a last explicit example of
Cremona equivalence for reduced schemes:  the case of
points.

This is the only case  in which I am able to provide the Cremona modification
in a
quite explicit way.

\begin{con}\label{con:enoughdjen}
  Let us consider a \dJen transformation of degree $d$,
  $$\omega:\p^r\to \p^r$$
given by
$$[x_0,\ldots,x_r]\mapsto[x_0F_0+G_0,x_1(x_0F+G),\ldots,x_r(x_0F+G)]$$
Then $p_0=[1,0,\ldots,0]$ is the vertex of all the monoids and the
lines through $p_0$ are preserved. Let $l\ni p_0$ be a line and assume
that $\omega$ is defined on the generic point of $l$. Then we have
that either $\omega_{|l}$ is an automorphism or $\omega(l)=p_0$.
Moreover $\omega$ is an isomorphism outside the cone with vertex $p_0$
and base  $$(x_0F_0+G_0=x_0F+G=0).$$

As a birational map, we can write $\omega$ as
$$[x_0,\ldots,x_r]\mapsto[\frac{x_0F_0+G_0}{x_0F+G}, x_1,\ldots,x_r]. $$
Let $p=[a_0,\ldots,a_r], q=[b_0,\ldots,b_r]\in\p^r\setminus\{p_0\}$ be points aligned with
$p_0$. Then we may assume that  $a_i=b_i$, for $i=1,\ldots,r$. Hence the condition
$\omega(p)=q$ translates into the equation
$$a_0F_0(a_1,\ldots,a_r)+G_0(a_1,\ldots,a_r)=b_0(a_0F(a_1,\ldots,a_r)+G(a_1,\ldots, a_r)), $$
linear in the coefficient of $F_0,G_0,F,G$.
Moreover if we choose a map $\omega$ such that $\omega(p)=p$, for a
point $p\in\p^r\setminus\{p_0\}$ then
$\omega$ is an isomorphism in a neighborhood of $p$.

Let us  pick two points, $p,q\subset\p^r$
and a set of $a$ points $\{p_1,\ldots,
p_a\}\subset\p^r\setminus\langle p,q\rangle$. Then we may choose
$p_0\in\langle p,q\rangle\setminus\{p,q\}$ such that no pair of points in $\{p_1,\ldots,
p_a\}$ is aligned with $p_0$.  This shows that there is a $d(a)$ such that for $d\geq
d(a)$ there is a \dJen map, $\omega$, centered in $p_0$ of degree $d$ such that $\omega(p_i)=p_i$,
for $i=1,\ldots,a$ and $\omega(p)=q$. In particular  $\omega$ is an isomorphism in
a neighborhood of $\{p_1,\ldots, p_a,p,q\}$.
\end{con}

Let us take advantage of this construction to give an explicit version
of the Cremona equivalence between reduced sets of points.

\begin{theorem}\label{th:puntiok}
  Let $Z=\{p_1,\ldots,p_s\}$ and $Z^\prime=\{p_1^\prime,\ldots,p_s^\prime\}$ be reduced sets of $s$ points in $\P^r$. Then there
  exists a Cremona transformation $\omega:\P^r\rat\P^r$ such that
  $\omega$ is an isomorphism in a neighborhood of $Z$ and
  $\omega(Z)=Z^\prime$.
\end{theorem}
\begin{proof}
  Let us prove the statement via a recursive argument.
   We may assume, eventually after a generic quadro-quadric modification, that for any $i=1,\ldots, s$
  $$\langle p_i,p_i^\prime\rangle\cap\{p_1,\ldots,p_s,p_1^\prime\ldots,p_s^\prime\}=\{p_i,p_i^\prime\}.$$
 Then by Construction~\ref{con:enoughdjen} there is a \dJen map $\phi_1:\p^r\rat\p^r$ such that:
  \begin{itemize}
  \item[-] $\phi_1(p_1)=p_1^\prime$,
    \item[-] $\phi_1(p_i)=p_i$ and
  $\phi_1(p_i^\prime)=p_i^\prime$ for $i\geq 2$.
  \end{itemize}
   In particular
  $\phi_1$ is an isomorphism in a neighborhood of $Z\cup Z^\prime$.
  Set, recursively, $\phi_i:\p^r\rat\p^r$ a \dJen map such that:
  \begin{itemize}
  \item[-] $\phi_i(p_j^\prime)=p_j^\prime$, for $j<i$,
  \item[-]  $\phi_i(p_i)=p_i^\prime$,
  \item[-] $\phi_i(p_h)=p_h$ and $\phi_i(p_h^\prime)=p_h^\prime$ for $h>i$.
    \end{itemize}
    Then the composition
    $$\Phi:=\phi_r\circ\cdots\circ\phi_1 $$
realizes a Cremona equivalence between $Z$ and $Z^\prime$. 
\end{proof}
\begin{remark}
  \label{rem:puntiok} Note that Theorem~\ref{th:puntiok} proves the
 Theorem~\ref{th:main} for $r=2$. The next section we will devoted to extend it to arbitrary
 $r\geq 3$.  
\end{remark}

\section{Cremona equivalence for reduced schemes}
\label{sec:crem-equiv}

In this section $X$ and $Y$ will be reduced
schemes in $\p^r$. Let us start observing a useful way to consider a
birational relation between them.  The schemes $X$ and $Y$ are
birational if exists a smooth scheme $Z$ such that:
\begin{itemize}
\item[-] $Z$ has a number of connected components equal to the
  number of irreducible components of $X$ and $Y$;
  \item[-]  there are two base point free linear systems
$\L_X$ and $\L_Y$ such that the induced morphism
$\phi_{\L_X}:Z\to X$ and $\phi_{\L_Y}:Z\to Y$ are dominant and  birational.
\end{itemize}

Let ${\mathcal M}=\L_X+\L_Y$ be the linear system on $Z$ and
$\phi_{\mathcal M}:Z\to\p^N$ the associated map. We may consider $X$
and $Y$ as linear projections of $\phi_{\mathcal
  M}(Z)\subset\p^N$. This is essentially the reason we opted in
Definition~\ref{def:Cremona} to accept non coprime sets of polynomials.
With this trick we will be able to factorize a Cremona equivalence between
$X$ and $Y$ into steps associated to  monoids.

\begin{con}[Double projection]
\label{con:double}
Let $X\subset \P^ r$ be an irreducible monoid of degree $d$.  
Let $p_1, p_2\in X$ be two vertices, let 
$H_1,H_2$ be hyperplanes with $p_i\not \in H_i$, and  { consider the
stereographic projections of $X$ from $p_i$, which is the restriction of the projection
$\pi_i\colon \P ^r\dasharrow H_i$ from $p_i$, with
$i=1,2$.} The map
\[
\pi_{X,p_1,p_2}:=\pi_2\circ \pi_1^ {-1}\colon H_1\dasharrow H_2
\] is a Cremona transformation. If $p_1=p_2=p$, then
$\pi_{X,p,p}$ does not depend on $X$ and it is a linear
automorphism, classically called the perspective with center $p$ of $H_1$ to $H_{2}$.

From now on, we
restrict to the case when $p_1\neq p_2$. In this setting, the map
$\pi_{X,p_1,p_2}$ is called the double projection and it depends on $X$ and it is in general nonlinear.
Assume that   $p_r=[0,\ldots,0,1], p_{r-1}=[0,\ldots,
0,1,0]$ and the hyperplanes $H_1,  H_2$ have equations $(x_r=0)$ and
$(x_{r-1}=0)$ respectively.
Then the  defining equation of $X$ has the form
\begin{equation*}\label{eq:bimon}
F_d+x_{r-1}G_{d-1}+x_rF_{d-1}+ x_{r}x_{r-1}F_{d-2}=0,
\end{equation*}
with $F_i,G_i\in\C[x_0,\ldots,x_{r-2}]_i$.
Then  the double projection map $\pi_{X,p_r,p_{r-1}}$ is given by 
\[
[x_0,\ldots, x_{r-1}]\mapsto [(F_{d-2}x_{r-1}+F_{d-1})x_0,\ldots,
(F_{d-2}x_{r-1}+F_{d-1})x_{r-2},-F_d-x_{r-1}G_{d-1}].
\]
Observe that the double projection is a \dJen map of degree $d$
centered in $p_{r-1}\in H_1$.
  \end{con}

The main idea to produce the Cremona equivalence between $X$ and $Y$
is borrowed from \cite{MP09}. Since $X$ and $Y$ are linear projection
of the same variety $\phi_{\mathcal M}(Z)\subset\p^N$ their embedding
is determined by functions on $Z$ that are linearly equivalent. Let us
see how to use this remark. Let, in this set up,  $\phi:Z\rat X\subset\p^r$ be given
by equations
$$t\mapsto [\phi_0(t),\ldots,\phi_r(t)]$$
and $\psi:Z\rat Y\subset\p^r$ by equations
$$t\mapsto[\psi_0(t),\ldots,\psi_r(t)],$$
with $t$ coordinates on a dense open subset of  $Z$ intersecting all
connected components.
In general $\{(\phi_i=0)\}$ and $\{(\psi_j=0)\}$ have fixed divisorial
component but nonetheless they define birational maps to $X$ and $Y$ respectively. 

Then we may consider the birational embedding $\eta:Z\rat Z_1\subset\p^{r+1}$
given by equations
$$t\mapsto[\phi_0(t),\ldots,\phi_r(t),\psi_0(t)].$$

Assume that there is an irreducible monoid $X$, with vertices $p_r$ and $p_{r+1}$
and containing
$Z_1$. Then the double projection $\pi_{X,p_{r+1},p_r}$ produces a Cremona
map $\omega:\p^r\rat\p^r$
such that
$\omega(X)$ is associated to the birational embedding
$$\phi_1:Z\rat
X_1\subset\p^r$$
given by equations
$$t\mapsto[\phi_0(t),\ldots,\phi_{r-1}(t),\psi_0(t)].$$
If further, the monoid $X$ does not contains any of the cones with vertex
either $p_r$ or $p_{r+1}$
and base an irreducible component of $Z_1$, then the double projection realizes a
Cremona equivalence between $X$ and $\omega(X)$.

Iterating this
process we may substitute the functions $\phi_j$ with the functions
$\psi_h$ realizing a chain of double projections, that is \dJen maps, that produces the
required Cremona equivalence.

To let this argument work we need to produce the required monoids.
Let us start rephrasing \cite[Lemma 2.1]{CCMRZ16} to the reducible case, I adopt notation of \cite[Chapter
6]{Fu98} for the intersection theory needed.
\begin{remark}
  I want to thank the referee for pointing out the following
 version of the proof that allows to remove the assumption that the
 projection of $Z$ from $p_r$ is birational.
\end{remark}

\begin{lemma}
  \label{lem:moniddim}
  Let $Z:=\cup_1^h Z_i\subset\p^r\setminus\{[0,\ldots,0,1]\}$ be a reduced scheme, $M_d$ the
  linear system of monoids with vertex $p_r:=[0,\ldots,0,1]$ and
  $M(Z)_d\subset M_d$
  those  containing the scheme $Z$.
  Then, for $d\gg0$,   we have
  $$\dim (M(Z)_d)\geq
  \frac{2d^{r-1}}{(r-1)!}+\frac{(r-1-\delta)d^{r-2}}{(r-2)!}+O(d^{r-3}), $$
  where
  $$\delta=Z\cdot \O_{Z}(1)^{r-2},$$
 in particular $\delta=0$ if  $\dim Z<r-2$.
\end{lemma}
\begin{proof}    Let $\nu:V\to \P^ r$ be the  blow--up of $\P^ r$ at
  $p_r$ with exceptional divisor $E$. We denote by $H$ the pull back  of a general hyperplane of $\P^ r$ and by $Z'$ the proper transform of $Z$. 

  In this notation we have
  $$M_d \cong\vert dH-(d-1)E\vert= \vert (d-1)(H-E)+H\vert,$$
and, by a simple dimension count,
\begin{equation}\label{eq:mon}
\dim (M_d)=\frac{2d^{r-1}}{(r-1)!}+\frac{(r-1)d^{r-2}}{(r-2)!}+O(d^{r-3}).
\end{equation}
Set $s=\dim Z\leq r-2$ by assumption the point $p\not\in Z$, hence
$$\o_V(E)\otimes\o_{Z^\prime}\cong\o_{Z^\prime} $$
and
$$\o_{Z^\prime} (dH-rE)\cong\o_{Z^\prime} (dH)\cong\o_{Z}(d),  $$
for any $d,r\in\Z$.
In particular, as a polynomial in $d$
$$h^0(\o_{Z^\prime}(dH-rE))=h^0(\o_{Z}(d))=\frac{\delta}{s!}d^s+o(d^{s-1}).$$
Thus
$$\dim(M(Z)_d)\geq \dim(M_d)-h^0(\o_{Z^\prime}(d(H-E)))= \frac{2d^{r-1}}{(r-1)!}+\frac{(r-1-\delta)d^{r-2}}{(r-2)!}+O(d^{r-3}). $$

\end{proof}

Next we use Lemma~\ref{lem:moniddim}  to produce monoids.

\begin{lemma}  \label{lem:mp16_1} Let $Z=\cup_1^h Z_j\subset \P^ r$, with
  $r\geqslant 3$, be a reduced scheme of  dimension
  $r-2$ and let $p\in \P^ r\setminus Z$ be  such that the projection
  of $Z$ from $p$ is birational to its image. For $d\gg 0$ there is an irreducible
  monoid of degree $d$ with vertex $p$, containing $Z$ and not
  containing the cone  $C_p(Z_j)$
  over $Z_j$ with vertex $p$, for $j=1,\ldots h$.
\end{lemma}

\begin{proof} 
  In the notation of Lemma~\ref{lem:moniddim} consider
$M(Z)_d\subset M_d$ the sublinear system of monoids containing $Z$.

By Lemma~\ref{lem:moniddim} we have
$$\dim M(Z)_d\geq
\frac{2d^{r-1}}{(r-1)!}+\frac{(r-1-\delta)d^{r-2}}{(r-2)!}+O(d^{r-3})>0, $$
where $\delta$ is the degree of the
$(r-2)$-dimensional part of $Z$. Note that $r\geq 3$ forces $\delta>0$.

   \begin{claim}\label{cl:monoidexists} For any $j=1,\ldots, s$ and
     $d\gg0$ there is a monoid $B_j\in M(Z_j)_d$ such that
     $$ B_j\not\supset C_p(Z_j).$$
   \end{claim}
   \begin{proof}
     Let $a=\dim Z_j$ and $\pi_j:\p^{r}\rat\p^{a+2}$ be a general linear
     projection, if $a=r-2$ we set $\pi_j=id_{\p^r}$.
     Set $\overline{p}:=\pi_j(p)$,
     $\overline{Z}:=\pi_j(Z_j)$ and $\alpha=\deg\overline{Z}$.

     To prove the claim it is enough to produce a monoid
     in $\p^{a+2}$  with vertex $\overline{p}$, containing
     $\overline{Z}$ and not containing the cone
     $Y:=C_{\overline{p}}(\overline{Z})$.
     
Let $\overline{M}(\overline{Z})_d$ be the linear system of monoids with
vertex $\overline{p}$ in $\p^{a+2}$ and containing $\overline{Z}$.

By Lemma~\ref{lem:moniddim} we have
$$\dim \overline{M}(\overline{Z})_d\geq
\frac{2d^{a+1}}{(a+1)!}+\frac{(a+1-\alpha)d^{a}}{a!}+O(d^{a-1})>0. $$

Let $M'\subset \overline{M}(\overline{Z})_d$ be the sublinear system
of divisors
 containing the cone $Y$. Note that
 $Y$ is a
 hypersurface of degree $\alpha$, i.e. $Y\in|\O_{\p^{a+2}}(\alpha)|$.
 Hence we have  $M'\cong M_{d-\alpha}$ and 
$$
 \dim (M')=  \frac{2(d-\alpha)^{a+1}}{(a+1)!}+\frac{(a+1)(d-\alpha)^{a}}{a!}+O(d^{a-1})=
\frac{2d^{a+1}}{(a+1)!}+\frac{(a+1-2\alpha)d^{a}}{a!}+O(d^{a-1}).
$$
Hence 
$$\dim(\overline{M}(\overline{Z})_d)-\dim(M')=\frac{\alpha
  d^{a}}{a!}+O(d^{a-1})>0, \quad \text{for}\quad d\gg 0.$$
This shows the existence of the required monoids.
   \end{proof}

 Set
 \begin{itemize}
 \item[-] $\pi:\p^{r}\rat\p^{r-1}$ the projection from $p$
 \item[-] $\tilde{Z}:=\pi(Z)$
   \item[-] $\tilde{Z}_j=\pi(Z_j)$, for any irreducible component
     $Z_j\subset Z$.
   \end{itemize}
   By hypothesis for any $j=1,\ldots,h$ the variety  $\tilde{Z}_j$ is
   an irreducible component of degree $\deg Z_j$ 
   of $\tilde{Z}$. In particular $Z_j$ is not contained in the cone
   over $\tilde{Z}\setminus\tilde{Z_j}$ with vertex $p$. Let
   $$\tilde{D}_j\in|\I_{\tilde{Z}\setminus\tilde{Z}_j}(d)|$$
   be a divisor in $\p^{r-1}$ of degree $d$ containing
   $\tilde{Z}\setminus\tilde{Z}_j$, and
   $D_j=C_p(\tilde{D})$ its cone with vertex $p$.

By Claim~\ref{cl:monoidexists}, for $d\gg0$,  we have $D_j+B_j\in M(Z)_{2d}$. Moreover
   $D_j+B_j$ does not contain the cone $C_p(Z_j)$. This shows that the
   general element in $M(Z)_{2d}$ does not contain  $C_p(Z_j)$.
   Hence the general element in $M(Z)_{2d}$ does not contain any of the
   cones $C_p(Z_j)$, for $j=1,\ldots,s$. Note that a reducible monoid
   decomposes in the union of cones, with vertex $p$, and a single
   irreducible monoid. Therefore our construction shows that  the general
   element in $M(Z)_{b}$ contains an irreducible monoid $X$ with
   $X\supset Z$ and $X\not\supset C_p(Z_j)$ for $j=1,\ldots,h$, for $b\gg0$.
\end{proof}

The next step is to produce the required double projections.

\begin{lemma} \label{lem:mp16} Let $Z=\cup_1^hZ_j\subset \P^ r$, with
  $r\geq 3$, be a reduced 
  scheme of positive dimension $n\leqslant r-3$. Let $p_1,p_2\in
  \P^ r\setminus Z$ be distinct points such that the projection of $Z$ from the
  line $\langle p_1,p_2\rangle$ is birational to its image. For
  $d\gg 0$ there is an irreducible monoid of degree $d$ with vertices $p_1$ and
  $p_2$, containing $Z$ but not containing any cone $C_{p_i}(Z_j)$,
  for $i=1,2$ and $j=1,\ldots, h$.
\end{lemma}

\begin{proof}  
We start the proof with a reduction to
codimension 3 subvarieties.

\begin{claim} It suffices to prove the assertion for $n=r-3$.
\end{claim} 

\begin{proof} [Proof of the Claim] Consider the projection of $\P^ r$
  to $\P^ {n+3}$ from a general linear subspace $\Pi$ of dimension
  $r-n-4$
and call $Z', p'_1,p'_2$ the projections of $Z,p_1,p_2$
respectively. Then $Z'$ is birational to $Z$ and it is still true that
the projection of $Z'$ form $\Span{p_1',p_2'}$ is birational to its image. The
dimension of $Z'$ is $n-3$.

Assume   the assertion holds for $Z', p'_1,p'_2$ and let $F'\subset
\P^ {n+3}$ be an irreducible monoid of degree $d\gg 0$ with vertices
$p'_1,p'_2$ containing $Z'$ but no irreducible components of  $C_{p_i^\prime}(Z')$, for $i=1,2$.
Let $F\subset \P^ r$ be the cone over $F'$ with vertex $\Pi$. Then $F$ is
an irreducible monoid with vertices $p_1,p_2$ containing $Z$ and no
irreducible component of $C_{p_i}(Z)$, for $i=1,2$.
\end{proof}

We can thus assume from now on that $n=r-3$. Fix two hyperplanes $H_1$
and $H_2$, where $p_1\not\in H_1$ and $p_2\not\in H_2$.  Let $Z^1$ and
$Z^2$ be the birational projections of $Z$ from $p_1$ and $p_2$ to $H_1$ and $H_2$,
respectively. Set $p'_{3-i}:=\pi_i(p_{3-i})$, for $i=1,2$. 
By hypothesis the projection of $Z^i$ from $p^\prime_{3-i}$ is
birational, then, by Lemma \ref
{lem:mp16}, for $i=1,2$ there are irreducible monoids $X_i\subset H_i$ with
vertex $p'_{3-i}$ such that:
\begin{itemize}
\item[-] $X_i\supset Z^i$
  \item[-] $X_i$ does not contain any irreducible component of
    $C_{p^\prime_{3-i}}(Z^i)\subset H_i$,
  \end{itemize}
  Set $Y_i:=C_{p_i}(X_i)\subset\p^r$ to be the cone over $X_i$ with vertex $p_i$, then $Y_i$ has the following properties:
  \begin{itemize}
\item[-] $Y_i$ is a cone with vertex $p_i$
  \item[-] $Y_i$ is a monoid with vertex $p_{3-i}$,
  \item[-] $Y_i$ contains the cone $C_{p_i}(Z)$,
    \item[-] $Y_i$ does not contain any irreducible component of the
      cone $C_{p_{3-i}}(Z)$.
    \end{itemize}
  Then a general linear combination of $Y_1$ and $Y_2$ contains an irreducible
  monoid with vertices $p_1$ and $p_2$ containing $Z$ and not
  containing any irreducible component of the cones with vertex $p_1$
  and $p_2$ over $Z$.
\end{proof}

To conclude the proof of Theorem~\ref{th:main} we will 
provide, for $r\geq 3$, the Cremona equivalence via a sequence of double projections
as in \cite[Theorem 1]{MP09} and  \cite[Theorem 2.5]{CCMRZ16}.
 To do this we plan to  use
Lemma~\ref{lem:mp16}. Therefore we need to ensure that projection from
the line connecting the 
vertices of monoids are birational. In \cite[Theorem 2.5]{CCMRZ16}
this was done via \cite[Theorem 1]{CC01}.
Let me spend a few word on this nice Theorem.

Let $X\subset\p^N$ be a non degenerate scheme, the Segre locus of $X$,
${\mathcal S}(X)$, is
the  locus  of points from which $X$ is projected in a non birational way. When $X$ is irreducible and reduced Calabri
and Ciliberto, \cite{CC01}, proved that this locus has irreducible components of
dimension less than $\dim X$, giving a very precise description of
its irreducible components. Unfortunately when $X$ is reducible this
is no more true. As an example of this behavior consider  $X=\cup L_i$ a union of
lines, with  $L_1\cap L_2\neq\emptyset$. Then any point of the plane
spanned by $L_1$ and $L_2$ is in the Segre locus of $X$.

Therefore the Segre locus of reducible schemes
is not well behaved as the one of irreducible varieties and therefore
we cannot adapt \cite[Theorem 1]{CC01} for our purposes and we need
to substitute it with a finer analysis than the one in \cite{MP09} of the individual steps of
the process.
The following is what we need to complete the proof of Theorem~\ref{th:main}.

\begin{theorem}\label{thm:mp} Let  $X, Y\subset \P^r$, with
  $r\geqslant 3$, be two reduced schemes of positive dimension $n<r-1$.  Then
  $X,Y$ are Cremona equivalent  if and only if they are birationally
  equivalent. \end{theorem}

\begin{proof} One direction is clear. Assume that $X$ and $Y$ are birational. Then, as
  observed at the beginning of the section there is a smooth scheme
  $Z$ and two birational morphisms
  $$\phi_{\L_X}:Z\to X\subset\p^r \ \ {\rm and}\ \ 
  \phi_{\L_Y}:Z\to Y\subset\p^r,$$
associated to linearly equivalent linear systems $\L_X\sim\L_Y$.
 
\begin{claim}\label{cl:projok} We may fix basis of $\L_X$ and $\L_Y$ such that the
  projection of  $X=\phi_{\L_X}(Z)$ and $Y=\phi_{\L_Y}(Z)$ from any coordinate subspace of dimension $m$ is
  birational to its image if $r> n+m+1$ and dominant to $\P^{r-m-1}$
  if $r\le n+m+1$.
\end{claim}

\begin{proof}[Proof of the Claim] It is well known that for any
   reduced scheme $X\subset \p^r$ of dimension $n$ the projection from a general linear space of
  dimension less than $r-n-1$ is birational and finite and the projection from a
  general space of dimension $r-n-1$  is  finite. Then it is
  enough to choose a basis of $\L_X$ and $\L_Y$ in such a way that such linear
  spaces  are coordinate subspaces.
\end{proof} 

 We may assume that $\phi_{\L_X}$
  is  given by equations
\[
t\mapsto[\phi_o(t),\ldots,\phi_r(t)]
\]
and $\phi_{\L_Y}$ is given by equations
\[
t\mapsto[\psi_0(t),\ldots,\psi_r(t)]
\]
where $(\phi_j=0), (\psi_h=0)$ are linearly equivalent divisors  on
$Z$ and $t$ varies in a suitable dense open subset of $Z$ intersecting
all irreducible components of $Z$. 

We prove the theorem by constructing a sequence of birational maps
$$\varphi_i: Z\dasharrow Z_i\subset \P^{r+1},$$ 
and projections
$$\eta_i:Z_i\dasharrow X_i,\ \ \nu_i:Z_i\dasharrow X_{i+1},$$
for $0\leqslant i\leqslant  r$, such that:\par
\begin{itemize} 
\item [(a)] $\eta_0\circ \varphi_0=\phi$ and $\nu_r\circ\varphi_r=\psi$, thus $X_0=X$ and $X_{r+1}=Y$;\par
\item [(b)] for $0\leqslant i\leqslant r$, there is a Cremona
  transformation $\omega_i\colon\P^ r\dasharrow \P^ r$, such that
  $\omega_i$ is an isomorphism in a neighborhood of the generic points
  of $X_i$, it satisfies $\omega_i(X_i)=X_{i+1}$
 and
 $\omega_{i|X_i}=\nu_i\circ\eta_i^{-1}$.
 \end{itemize}

 We may summarize the sequence of maps in the following diagram
    \[
 \xymatrix{
   &Z_0\ar_{\eta_0}[dl]\ar^{\nu_0}[dr]&&Z_1\ar_{\eta_1}[dl]\ar^{\nu_1}[dr]\ar@{.>}[rr]&&Z_r\ar^{\nu_r}[dr]&  \\
 X=X_0\ar@{.>}_{\omega_{0|X_0}}[rr]&              &X_1\ar@{.>}_{\omega_{1|X_1}}[rr]&&X_2\ar@{.>}[rr]&&X_{r+1}=Y.}
 \]

 The construction is done recursively.
For $i=0$ we set 
\[
\varphi_0(t)= [\phi_{0}(t),\ldots,\phi_{r}(t), \psi_{0}(t)],
\]
$$\eta_0:= \pi_{[0,\ldots,0,1]|Z_0}$$ the restriction of the projection
from the $(r+1)^{\rm th}$-coordinate point $p_{r+1}:=[0,\ldots,0,1]$ and
$$\nu_0:=\pi_{[0,\ldots,0,1,0]|Z_0}$$
the restriction of the projection
from the $r^{\rm th}$-coordinate point $p_r:=[0,\ldots,0,1,0]$.
By Claim~\ref{cl:projok} the projection from $\langle
p_r,p_{r+1}\rangle$ is birational. Then by Lemma~\ref{lem:mp16} there is a monoid $W\subset\p^{r+1}$
containing $Z_0$ and with vertices in $p_{r+1}$ and $p_r$ such that the
double projection $\pi_{W,p_{r+1},p_{r}}$  is an isomorphism on the generic points of $X_0$ and
realizes a Cremona equivalence $\omega_0:\p^r\to\p^r$ such that $\omega_0(X_0)=X_1$.

Assume $0< i\leqslant r-1$. In order to perform the step from
$i$ to $i+1$, we have to define the maps $\varphi_{i+1}$,
$\eta_{i+1}$, $\nu_{i+1}$ and $\omega_{i+1}$.
From the $i^{\rm th}$-step we have the map
$$\nu_i\circ\varphi_{i}: Z\dasharrow
X_{i+1}\subset \p^ {r}$$ given by 
$$t\mapsto [\tilde \phi_{i,0}(t),\ldots, \tilde
\phi_{i,r-i}(t),\psi_{0}(t),\ldots, \psi_i(t)],
$$
for suitable functions $\tilde{\phi}_{i,j}$.
Then we define $\varphi_{i+1}:Z\dasharrow Z_{i+1}\subset\p^{r+1}$ as
$$t\mapsto [\tilde \phi_{i,0}(t),\ldots, \tilde\phi_{i,r-i}(t),\psi_{0}(t),\ldots, \psi_i(t),\psi_{i+1}(t)].
$$
Note that we added the function $\psi_{i+1}$ on the last coordinate,
therefore 
$$\pi_{[0,\ldots,0,1]}(Z_{i+1})=X_{i+1}.$$ 
Therefore we set
$$\eta_{i+1}:=\pi_{[0,\ldots,0,1]|Z_{i+1}}.$$

To define $\nu_{i+1}$ we need to take a point
$$p\in\Pi_{i}:=  \{x_{r-i+1}=\ldots
=x_{r+1}=0\}\subset\p^{r+1}$$
and prove that the projection from the line $l_p:=\langle
p,[0,\ldots,0,1]\rangle$ restricts to a birational map on $Z_{i+1}$.

\begin{claim}\label{sub:claim1} The projection of $Z_{i+1}$ from a
  general line $l_p$ is birational to its
  image. \end{claim}

\begin{proof} [Proof of the Claim]
  Let  $\pi:=\pi_{[0,\ldots,0,1]}:\p^{r+1}\dasharrow
  \p^{r}$ be the projection from the point $[0,\ldots,0,1]$ and $\tilde{Z}=\pi(Z_{i+1})$.
  By construction $\pi_{|Z_{i+1}}$ is birational.
    Let 
  $$ A:=\{x_{r-i+1}=\ldots
  =x_{r}=0\}\subset\p^{r}$$
  be the linear space we are interested in and $\pi_A:\p^r\rat
  \p^{i-2}$ the linear projection from $A$. To prove the claim we have
  to prove that the projection from a general point of $A$ restricts
  to a birational map onto $\tilde{Z}$.
  
  By Claim~\ref{cl:projok} the restriction  $\pi_{A|\tilde{Z}}$ is  either birational onto the image or
  dominant.   If $\pi$ is birational the claim is clear.

  Assume that $\pi_{A|\tilde{Z}}$ is  dominant, in particular  $i-2\leq
  n$. Let $F\subset
  \tilde{Z}$ be a general fiber of this projection.
We have  $\dim
  F_j\leq n-i+2$, for all irreducible components $F_j\subset
  F$. Moreover  the
  fiber $F$ is contained in a linear space $P_F$ of dimension
  $r-i+2$ and $A\cap P_F$ is a hyperplane.
 Since 
  $$\cod_{P_F} F\geq 3$$
 the general projection from a line in $P_F$ restricts to $F$ as a
 birational map and being $A$ an hyperplane this shows that the
 projection, say $\pi_p$, from a general point $p\in A$ restricts to $F$ as a
 birational map.  Let $x\in F$ be a general point and $r$ the line
 spanned by $p$ and $x$. By construction we have
$$r\cap\tilde{Z}=r\cap
 F=\{x\}.$$
 The scheme $F$ is the general fiber of the
 linear projection $\pi$ and $x\in F$ is a general point, hence  the line $r$ is not tangent to $\tilde{Z}$ in $x$. This
 shows that the morphism $\pi_{p|\tilde{Z}}$ is birational  as required.
  \end{proof}

Let $p\in\Pi_{i}$ be a general point and $\pi_{p}:\p^{r+1}\dasharrow
\p^r$ the projection from $p$.
Set
$$\nu_{i+1}:=\pi_{p|Z_{i+1}}. $$
Thanks to Claim~\ref{sub:claim1} we are in the condition to apply
Lemma~\ref{lem:mp16} and produce  a monoid $W\subset\p^{r+1}$ with the
following properties:
\begin{itemize}
\item[-] $W\supset Z_{i+1}$
\item[-] $p_{r+1}$ and $p$ are vertices of $W$
  \item[-]  the double projection $\pi_{W,p_{r+1},p}$   is an
    isomorphism on the generic points of $X_{i+1}$.
  \end{itemize}
Therefore the  double projection $\pi_{W,p_{r+1},p}$ realizes a
birational map
$\omega_{i+1}:\p^r\rat\p^r$ such that $\omega_{i+1}(X_{i+1})=X_{i+2}$
and $\omega_{i+1}$ is an isomorphism in a neighborhood of the generic
points of $X_{i+1}$.
This proves part (b) of the requirements.

To conclude observe that at the $r^{\rm th}$-step we have
$$\varphi_{r}(t)=[\tilde{\phi}_{r,0}(t),\psi_0(t),\ldots,\psi_r(t)], $$
therefore, thanks to Claim~\ref{cl:projok}, the restriction of the  projection from $p_0:=[1,0,\ldots,0]$ is automatically
birational and the same is true for
the projection from the line
$\langle[1,0,\ldots,0],[0,\ldots,0,1]\rangle$.
Therefore we set 
$\nu_r:=\pi_{[1,0,\ldots,0]|Z_r}\circ\varphi_r$
to fulfill also the last part of requirement (a).

This chain of double projections realizes the Cremona equivalence
between $X$ and $Y$.
\end{proof}
\begin{remark}\label{rem:wild}
  It is interesting to stress the following point. We already observed
  that double projections are associated to \dJen Cremona
  transformations. Therefore  all Cremona equivalences of subvarieties of codimension
  at least $2$ can be realized by transformations in the subgroup generated by \dJen
  transformations. This is particularly interesting when confronted
  with  \cite{BLZ19} where it is proven that the group of
  \dJen map is a proper subgroup of $Cr_n$.  That is to produce all
  Cremona equivalences for codimension $\geq 2$
  reduced schemes we do not even need the full group $Cr_n$.
\end{remark}

As an application of the main result we prove a general
contractibility criteria for sets of rational varieties.

\begin{corollary}
  Let $Z=\cup_1^s T_i\subset \p^r$ be a reduced scheme all of whose irreducible
  components are rational varieties of dimension at most $r-2$. Then there is a
  birational map $\omega:\p^r\dasharrow \p^r$ that contracts $Z$ to a
  set of
  $s$ distinct points.
\end{corollary}
\begin{proof} By Theorem~\ref{thm:mp} there is a birational map
  $\phi:\p^r\dasharrow \p^r$ such that $\phi(Z)$ is a union of $s$
  linear spaces.
  We are therefore left to study the case of linear spaces.
  We prove the claim by induction on the dimension of $Z$.
  Assume $\dim Z=1$ and
  $$Z=\cup_1^h l_i\cup\{p_{h+1},\ldots,p_s\},$$
    with $l_i$ lines.
 Then there is a birational map $\omega:\P^r\rat\p^r$
  such that $\cap_1^h\omega(l_i)=p$ is a general point and $\omega(p_i)=p_i$.
 Consider a quadro-quadric map $\phi:\p^r\rat\p^r$ centered in $p$ and  a general
  codimension $2$  quadric $Q$ intersecting the $h$ lines. Then $\phi$
  contracts the $h$ lines to a set of $h$ points.
  
  Assume that $\dim Z=i$ and, by induction, that the result is true for sets of linear spaces of
  dimension at most $i-1\leq r-3$.

  Set
  $$Z=\cup_1^h M_i\cup Z^\prime,$$
  with $M_i\cong\p^i$ and $\dim Z^\prime\leq i-1$.
  Fix a general point $p\in\P^r$, a general codimension $2$ quadric
  $Q\subset H\subset\P^r$ containing $h$ linear
  spaces $A_i\subset Q$ of dimension $i-1$. By Theorem~\ref{thm:mp}
  there is a Cremona equivalence between $Z$ and
  $$W:=\cup_1^h\langle A_i,p\rangle\cup Z^\prime. $$
  Let $\omega:\P^r\rat\p^r$ be the
quadro-quadric map with base locus $p\cup Q$.
Then by construction
$$\omega_{|\langle A_i,p\rangle}\cong \p^{i-1}. $$
Hence $\omega(W)$ is a union of linear spaces  of dimension at most $i-1$ and we
can conclude by induction that $Z$ is contractible.
\end{proof}
\bibliographystyle{amsalpha}
\bibliography{Biblio}

\newcommand{\etalchar}[1]{$^{#1}$}
\providecommand{\bysame}{\leavevmode\hbox to3em{\hrulefill}\thinspace}
\providecommand{\MR}{\relax\ifhmode\unskip\space\fi MR }
\providecommand{\MRhref}[2]{%
  \href{http://www.ams.org/mathscinet-getitem?mr=#1}{#2}
}
\providecommand{\href}[2]{#2}
\begin{thebibliography}{CCM{\etalchar{+}}16}

\bibitem[BLZ21]{BLZ19}
J\'{e}r\'{e}my Blanc, St\'{e}phane Lamy, and Susanna Zimmermann,
  \emph{Quotients of higher-dimensional {C}remona groups}, Acta Math.
  \textbf{226} (2021), no.~2, 211--318. \MR{4281381}

\bibitem[Can18]{Ca15}
Serge Cantat, \emph{The {C}remona group}, Algebraic geometry: {S}alt {L}ake
  {C}ity 2015, Proc. Sympos. Pure Math., vol.~97, Amer. Math. Soc., Providence,
  RI, 2018, pp.~101--142. \MR{3821147}

\bibitem[CC01]{CC01}
Alberto Calabri and Ciro Ciliberto, \emph{On special projections of varieties:
  epitome to a theorem of {B}eniamino {S}egre}, Adv. Geom. \textbf{1} (2001),
  no.~1, 97--106. \MR{1823955}

\bibitem[CC10]{CC10}
\bysame, \emph{Birational classification of curves on rational surfaces},
  Nagoya Math. J. \textbf{199} (2010), 43--93. \MR{2730411}

\bibitem[CC17]{CC17}
\bysame, \emph{On the {C}remona contractibility of unions of lines in the
  plane}, Kyoto J. Math. \textbf{57} (2017), no.~1, 55--78. \MR{3621779}

\bibitem[CC18]{CC18}
\bysame, \emph{Contractible curves on a rational surface}, Local and global
  methods in algebraic geometry, Contemp. Math., vol. 712, Amer. Math. Soc.,
  Providence, RI, 2018, pp.~9--31. \MR{3832397}

\bibitem[CCM{\etalchar{+}}16]{CCMRZ16}
Ciro Ciliberto, Maria~Angelica Cueto, Massimiliano Mella, Kristian Ranestad,
  and Piotr Zwiernik, \emph{Cremona linearizations of some classical
  varieties}, From classical to modern algebraic geometry, Trends Hist. Sci.,
  Birkh\"{a}user/Springer, Cham, 2016, pp.~375--407. \MR{3776660}

\bibitem[CE00]{CE00}
Guido {Castelnuovo} and Federigo {Enriques}, \emph{{Sulle condizioni di
  razionalit\`a dei piani doppi}}, {Rend. Circ. Mat. Palermo} \textbf{14}
  (1900), 290--302 (Italian).

\bibitem[CL13]{CL13}
Serge Cantat and St\'{e}phane Lamy, \emph{Normal subgroups in the {C}remona
  group}, Acta Math. \textbf{210} (2013), no.~1, 31--94, With an appendix by
  Yves de Cornulier. \MR{3037611}

\bibitem[{Coo}59]{Co59}
Julian~Lowell {Coolidge}, \emph{{A treatise on algebraic plane curves.
  Unabridged and unaltered republ. of the first ed}}, {New York: Dover
  Publications, Inc. XXIV, 513 p. (1959).}, 1959.

\bibitem[Dur19]{Du19}
Sara Durighetto, \emph{Classical and derived birational geometry}, Phd Thesis,
  accessible at
  https://drive.google.com/file/d/1Jxa3zxWAh9PhbtJ5LkhiJzhWBgAKfciQ/view, 2019.

\bibitem[Ful98]{Fu98}
William Fulton, \emph{Intersection theory}, second ed., Ergebnisse der
  Mathematik und ihrer Grenzgebiete. 3. Folge. A Series of Modern Surveys in
  Mathematics [Results in Mathematics and Related Areas. 3rd Series. A Series
  of Modern Surveys in Mathematics], vol.~2, Springer-Verlag, Berlin, 1998.
  \MR{1644323}

\bibitem[{Hud}27]{HU27}
Hilda~P. {Hudson}, \emph{{Cremona transformations in plane and space}}, {XX +
  454 p. Cambridge, University Press (1927).}, 1927.

\bibitem[Isk87]{Is87b}
Iskovskikh, \emph{Cremona group}, Encyclopedia of Mathematics, 1987.

\bibitem[Jel87]{Je87}
Zbigniew Jelonek, \emph{The extension of regular and rational embeddings},
  Math. Ann. \textbf{277} (1987), no.~1, 113--120. \MR{884649}

\bibitem[{Mar}07]{Ma07}
Giuseppe. {Marletta}, \emph{{Sulla identit\`a cremoniana di due curve piane}},
  {Rend. Circ. Mat. Palermo} \textbf{24} (1907), 229--242 (Italian).

\bibitem[Mel13]{Me13}
Massimiliano Mella, \emph{Equivalent birational embeddings {III}: {C}ones},
  Rend. Semin. Mat. Univ. Politec. Torino \textbf{71} (2013), no.~3-4,
  463--472. \MR{3506396}

\bibitem[Mel20]{Me20}
\bysame, \emph{Birational geometry of rational quartic surfaces}, J. Math.
  Pures Appl. (9) \textbf{141} (2020), 89--98. \MR{4134451}

\bibitem[Mel21]{Me21}
\bysame, \emph{The minimal cremona degree of quartic surfaces}, arXiv math.AG
  2105.12448, 2021.

\bibitem[MP09]{MP09}
Massimiliano Mella and Elena Polastri, \emph{Equivalent birational embeddings},
  Bull. Lond. Math. Soc. \textbf{41} (2009), no.~1, 89--93. \MR{2481993}

\bibitem[MP12]{MP12}
\bysame, \emph{Equivalent birational embeddings {II}: divisors}, Math. Z.
  \textbf{270} (2012), no.~3-4, 1141--1161. \MR{2892942}

\bibitem[PS15]{PS15}
Ivan Pan and Aron Simis, \emph{Cremona maps of de {J}onqui\`eres type}, Canad.
  J. Math. \textbf{67} (2015), no.~4, 923--941. \MR{3361019}

\bibitem[Ver05]{Ve05}
Alessandro Verra, \emph{Lectures on cremona transformations},
  http://www.mat.uniroma3.it/users/verra/CREMONA20.dvi, 2005.

\end{thebibliography}

\end{document}